\documentclass[12pt]{amsart}
\usepackage[margin=3.5cm]{geometry}
\usepackage{xcolor}

\usepackage{amsfonts, amsthm, amssymb,verbatim,amscd,amsmath, blindtext}

\usepackage{multirow}
\usepackage{graphicx}
\usepackage{enumitem}
\usepackage{amssymb}
\usepackage{ytableau}
\usepackage{tikz}
\usepackage{tikz-cd}
\usepackage{float, pifont}
\usepackage{mathtools}
\usepackage[pagebackref,colorlinks=true,citecolor=blue,linkcolor=black]{hyperref}
\usepackage{fullpage}

\newtheorem{theorem}{Theorem}[section]
\newtheorem{lemma}[theorem]{Lemma}
\newtheorem{corollary}[theorem]{Corollary}
\newtheorem{proposition}[theorem]{Proposition}
\newtheorem{conjecture}[theorem]{Conjecture}

\theoremstyle{definition}
\newtheorem{definition}[theorem]{Definition}
\newtheorem{example}[theorem]{Example}
\newtheorem{algorithm}[theorem]{Algorithm}

\newtheorem{remark}[theorem]{Remark}
\newtheorem{question}[theorem]{Question}

\makeatletter
\newenvironment{sqcases}{%
  \matrix@check\sqcases\env@sqcases
}{%
  \endarray\right.%
}
\def\env@sqcases{%
  \let\@ifnextchar\new@ifnextchar
  \left\lbrack
  \def\arraystretch{1.2}%
  \array{@{}l@{\quad}l@{}}%
}
\makeatother

\DeclareMathOperator{\rank}{rank}

\DeclareMathOperator{\lcm}{lcm}

\DeclareMathOperator{\pd}{pd}

\DeclareMathOperator{\sbridge}{sb}
\DeclareMathOperator{\Taylor}{Taylor}

\DeclareMathOperator{\characteristic}{char}
\DeclareMathOperator{\mingens}{MinGens}

\DeclareMathOperator{\dist}{dist}

\newcommand{\E}{\mathcal{E}}

\newcommand{\ZZ}{{\mathbb Z}}
\newcommand{\NN}{{\mathbb N}}

\newcommand{\card}[1]{{\left\lvert #1 \right\lvert}}

\newcommand{\sunletsymb}[1]{%
    \begin{tikzpicture}[scale=1, %
    thick,
    baseline=-1ex
    ]
    \def\x{1mm};
    \def\theta{360/#1};

    \foreach \i in {1,...,#1}{
        \draw ({\x*cos(\theta*\i)}, {\x*sin(\theta*\i)}) -- ({\x*cos(\theta*(\i+1))}, {\x*sin(\theta*(\i+1))});
        \draw ({\x*cos(\theta*\i)}, {\x*sin(\theta*\i)}) -- ({2*\x*cos(\theta*\i)}, {2*\x*sin(\theta*\i)});
    }
    \end{tikzpicture}%
}

\newcommand{\sixcycleconsecutivesymb}{%
    \begin{tikzpicture}[scale=1, %
    thick,
    baseline=-1ex
    ]
    \def\x{1mm};
    \def\theta{60};

    \foreach \i in {1,...,6}{
        \draw ({\x*cos(\theta*\i)}, {\x*sin(\theta*\i)}) -- ({\x*cos(\theta*(\i+1))}, {\x*sin(\theta*(\i+1))});
    }
    \foreach \i in {-1,0,1}{
        \draw ({\x*cos(\theta*\i)}, {\x*sin(\theta*\i)}) -- ({2*\x*cos(\theta*\i)}, {2*\x*sin(\theta*\i)});
    }
    \end{tikzpicture}%
}

\newcommand{\sixcyclealternatesymb}{%
    \begin{tikzpicture}[scale=1, %
    thick,
    baseline=-1ex
    ]
    \def\x{1mm};
    \def\theta{60};

    \foreach \i in {1,...,6}{
        \draw ({\x*cos(\theta*\i)}, {\x*sin(\theta*\i)}) -- ({\x*cos(\theta*(\i+1))}, {\x*sin(\theta*(\i+1))});
    }
    \foreach \i in {-1,1,3}{
        \draw ({\x*cos(\theta*\i)}, {\x*sin(\theta*\i)}) -- ({2*\x*cos(\theta*\i)}, {2*\x*sin(\theta*\i)});
    }
    \end{tikzpicture}%
}

\newcommand{\dotsymb}[1]{%
	\begin{tikzpicture}[scale=0.5, %
		thick,
		baseline=-0.5ex
		]
		\def\x{15mm};
		\def\r{0.8mm};
		
		\draw[fill=black] (0, 0) circle (\r);
	\end{tikzpicture}%
}

\newcommand{\smallestcounterexamplesymb}{%
	\begin{tikzpicture}[scale=0.8, inner sep=0pt, baseline=0ex]
		\def\x{7mm};
		\def\r{0.45mm};
		
	    \node (t) at (-2*\x, 0) {};
	    \node (u) at (0, \x) {};
	    \node (v) at (\x, 0) {};
	    \node (w) at (0, 0) {};
	    \node (x) at (-\x, 0) {};
	    \node (y) at (0, -\x) {};
	    \node (z) at (2*\x, 0) {};

	    \draw[] (t) to (x);
	    \draw[] (x) to (w);
	    \draw[] (u) to (v);
	    \draw[] (u) to (x);
	    \draw[] (v) to (w);
	    \draw[] (y) to (v);
	    \draw[] (y) to (x);
	    \draw[] (v) to (z);

	    \draw[fill=black] (t) circle (\r);
	    \draw[fill=black] (u) circle (\r);
	    \draw[fill=black] (v) circle (\r);
	    \draw[fill=black] (w) circle (\r);
	    \draw[fill=black] (x) circle (\r);
	    \draw[fill=black] (y) circle (\r);
	    \draw[fill=black] (z) circle (\r);
	\end{tikzpicture}%
}

\def\E{{\mathcal E}}

\def\H{{\mathcal H}}

\def\x{{\bf x}}

\def\1{{\bf 1}}
\def\0{{\bf 0}}
\def\r{{\mathbf r}^{\mathbb L}}

\begin{document}
\title{\textbf{Monomial ideals with minimal generalized Barile-Macchia resolutions}}
\author{Trung Chau}
\address{Chennai Mathematical Institute, Chennai, 603103, India}
\email{chauchitrung1996@gmail.com}

\author{T\`ai Huy H\`a}
\address{Tulane University, Mathematics Department, 6823 St. Charles Avenue, New Orleans, LA 70118, USA}
\email{tha@tulane.edu}

\author{Aryaman Maithani}
\address{Department of Mathematics, University of Utah, 155 South 1400 East, Salt Lake City, UT~84112, USA}
\email{maithani@math.utah.edu}

\maketitle

\begin{abstract}
   We identify several classes of monomial ideals that possess minimal generalized Barile-Macchia resolutions. These classes of ideals include generic monomial ideals, monomial ideals with linear quotients, and edge ideals of hypertrees. We also characterize connected unicyclic graphs whose edge ideals are bridge-friendly and, in particular, have minimal Barile-Macchia resolutions. Barile-Macchia and generalized Barile-Macchia resolutions are cellular resolutions and special types of Morse resolutions.
\end{abstract}



\section{Introduction}

Understanding when a monomial ideal admits a \emph{cellular resolution} and determining explicit descriptions of such resolutions are challenging problems that have been extensively studied (cf.~\cite{AFG2020, BM20, BW02, BPS98, BS98, CHM24-first, CK24, CT2016, CEFMMSS21, CEFMMSS22, Ly88, OY2015, Vel08}). Only a few general constructions exist, such as the \emph{Taylor resolution}, the \emph{Lyubeznik resolution}, the \emph{Morse resolution} and, in special cases, the \emph{Scarf complex} (see~\cite{BW02, BPS98, Ly88, Tay66}). 

In 2002, Batzies and Welker~\cite{BW02} applied discrete Morse theory to provide Morse resolutions of monomial ideals and, in particular, introduced what is now referred to as the \emph{generalized} Lyubeznik resolution. More recently, Chau and Kara~\cite{CK24}, building on a prior work of Barile and Macchia~\cite{BM20}, developed the Barile-Macchia and generalized Barile-Macchia resolutions, which are also special classes of Morse resolutions. However, determining when these resolutions are \emph{minimal} remains an important unresolved problem. Specifically, there is growing interest in finding classes of monomial ideals that admit minimal Taylor, generalized Lyubeznik, generalized Barile-Macchia, or Morse resolutions (see~\cite{CHM24-first, CKW24, CEFMMSS21, CEFMMSS22, FHHM24}).

In this paper, we identify several classes of monomial ideals that possess \emph{minimal generalized Barile-Macchia resolutions}. Our results parallel those of Batzies and Welker~\cite{BW02}, who showed that monomial ideals that are generic or have linear quotients admit minimal generalized Lyubeznik resolutions. The notion of \emph{generic} monomial ideals that we use came from~\cite{MSY00}, which is more inclusive than that given in~\cite{BPS98}. Most monomial ideals are generic, in the sense that they form a Zariski-open subset in the matrix space of exponents. On the other hand, monomial ideals with \emph{linear quotients} have been widely studied (cf.~\cite{HH2011Text} and references therein) due to their deep connections to combinatorial structures. Notably, for squarefree monomial ideals, the property of having linear quotients corresponds to the associated Stanley-Reisner simplicial complex being \emph{shellable}.
Leaving precise notations and terminology until later, our results are stated as follows.

\noindent\textbf{Theorems \ref{thm:generic-BM} and \ref{thm:shellable-BM}.} Let $S = \Bbbk[x_1, \dots, x_d]$ be a polynomial ring over a field $\Bbbk$ and let $I \subseteq S$ be a monomial ideal. Suppose that $I$ is either generic or has linear quotients. Then, $S/I$ has a minimal generalized Barile-Macchia resolution.

To prove Theorems \ref{thm:generic-BM} and \ref{thm:shellable-BM}, we show that in these cases the generalized Barile-Macchia resolution coincides with the generalized Lyubeznik resolution. This enables us to employ the results of~\cite{BW02}. As observed from discrete Morse theory, both the generalized Lyubeznik and the generalized Barile-Macchia resolutions are constructed from the Taylor resolution by identifying \emph{acyclic matchings} and \emph{critical sets}. The criteria for these acyclic matchings and critical sets are given in Theorems~\ref{thm:generalized-Lyubeznik} and~\ref{thm:generalized-BM}. The proofs are completed by showing that for generic monomial ideals and monomial ideals with linear quotients, the acyclic matchings and critical sets to construct the generalized Lyubeznik and generalized Barile-Macchia resolutions are the same; see Lemma~\ref{lem:Lyu=BM}.

Next, we focus on squarefree monomial ideals. These ideals can be viewed as the \emph{edge ideals} of graphs and hypergraphs. For graphs particularly, computational experiments show that edge ideals of graphs with at most 8 vertices have minimal generalized Barile-Macchia resolutions; see Theorem~\ref{thm:generalized-BM-graphs}. Furthermore, it is a consequence of our results in Section~\ref{sec.hypertree} --- see Theorem~\ref{thm:rooted-hypertree-friendly} below --- that the edge ideals of trees have minimal Barile-Macchia resolutions. An important class of connected graphs that are not trees consists of \emph{unicyclic} graphs, the graphs containing exactly one cycle. Our results identify unicyclic graphs whose edge ideals have minimal Barile-Macchia resolutions. To achieve this, we look at the stronger, but better manageable, property of being \emph{bridge-friendly}. We characterize the unicyclic graphs whose edge ideals are bridge-friendly. Again, leaving precise notations and terminology until later, our result is stated as follows.

\noindent\textbf{Theorem~\ref{thm:bridgefriendly-unicyclic}.} Let $G$ be a connected unicyclic graph. Then, $I(G)$ is bridge-friendly if and only if either 
\vspace*{-0.2cm}\begin{enumerate} 
\item $G$ contains a $C_3$ or a $C_5$ with one vertex of degree 2; or
\item $G$ contains a $C_6$ with two opposite vertices of degree 2.
\end{enumerate}

Theorem~\ref{thm:bridgefriendly-unicyclic} is proved in two steps. First, we exhibit a collection of ``forbidden structures'' for being bridge-friendly, i.e., small graphs whose edge ideals are not bridge-friendly; see Proposition~\ref{prop:non-bridgefriendly-graphs}. This, coupled with~\cite[Proposition 4.2]{CHM24-first}, shows that if $I(G)$ is bridge-friendly then $G$ must be of the prescribed forms in the statement of the theorem. The last step is to show that if $G$ is of one of the prescribed forms, then $I(G)$ is bridge-friendly. 

In the more general context of hypergraphs, our approach is based on the notion of \emph{host} graphs associated to a given hypergraph. The concept of host graphs arrives from optimization theory (cf.~\cite{Ber90, BDCV98}). A hypergraph is called a \emph{rooted hypertree} if it has a host graph that is a tree with the property that each edge of the hypergraph consists of vertices of different distances from a fixed vertex $x$ (the \emph{root}). The class of rooted hypertrees contains trees and rooted trees, and these have been much examined (cf.~\cite{BM20, BHO11,CK24}). 

Our last result shows that edge ideals of rooted hypertree has a minimal Barile-Macchia resolution.

\noindent\textbf{Theorem~\ref{thm:rooted-hypertree-friendly}.} Let $\H$ be a rooted hypertree. Then, its edge ideal $I(\H)$ has a minimal Barile-Macchia resolution.

The proof of Theorem~\ref{thm:rooted-hypertree-friendly} is based on a combinatorial analysis of the host graphs of hypergraphs. Particularly, the structure of the host tree of a rooted hypertree $\H$ allows us to define a rank function on the vertices of $\H$, which results in a total order of the minimal generators of $I(\H)$. The proof proceeds by showing that the Barile-Macchia resolution of $I(\H)$ with respect to this total order of the generators is minimal, making use of previous characterizations from~\cite{CHM24-first, CK24}.

The structure of the paper is as follows. In Section~\ref{sec.Morse}, we provide background on discrete Morse theory and the construction of generalized Lyubeznik and generalized Barile-Macchia resolutions. Section~\ref{sec.gen_shellable} focuses on monomial ideals that are generic or have linear quotients, proving Theorems~\ref{thm:generic-BM} and~\ref{thm:shellable-BM}. Section~\ref{sec.edgeIdeal} studies edge ideals of graphs, particularly unicyclic graphs and graphs with few vertices, establishing Theorems~\ref{thm:generalized-BM-graphs} and~\ref{thm:bridgefriendly-unicyclic}. Section~\ref{sec.hypertree} investigates rooted hypertrees and culminates with the proof of Theorem~\ref{thm:rooted-hypertree-friendly}.

\medskip

\noindent\textbf{Acknowledgment.} The first and third authors were partially supported by NSF grants DMS 1801285 and 2101671. The first author was also partially supported by the NSF grant DMS 2001368. The second author acknowledges support from a Simons Foundation grant. The third author made extensive use of the computer algebra systems \texttt{SageMath} \cite{sagemath} and \texttt{Macaulay2} \cite{M2}, and the package \texttt{nauty} \cite{nauty}; the use of these is gratefully acknowledged. 


\section{Discrete Morse Theory and Lyubeznik/Barile-Macchia Resolutions} \label{sec.Morse}

In this section, we give a brief overview of how to apply discrete Morse theory to the Taylor resolution of any monomial ideal to construct its generalized Lyubeznik and generalized Barile-Macchia resolutions. Throughout the section, $S = \Bbbk[x_1, \dots, x_d]$ denotes a polynomial ring over a field $\Bbbk$ and $I$ is a monomial ideal in $S$. Let $\Taylor(I)$ represent the Taylor resolution of $I$ (see~\cite{Tay66}).

Let $\mingens(I)$ be the set of minimal monomial generators of $I$, and let $\mathcal{P}(\mingens(I))$ be its power set. Observe that $S$ can be viewed as an $\NN^d$-graded ring, so by taking the $\ZZ^d$-degrees of monomials in $S$, the least common multiple operation defines a map on $\mathcal{P}(\mingens(I))$. We denote this map by $\lcm$, i.e.,
\begin{align*} 
\lcm \colon \mathcal{P}(\mingens(I)) &\longrightarrow \ZZ^d \\
\sigma &\mapsto \deg(\lcm(m ~\big|~ m \in \sigma)).
\end{align*}
Let $P$ be a \emph{poset} and let $f\colon \mathcal{P}(\mingens(I))\to P$ be an order-preserving map, where $\mathcal{P}(\mingens(I))$ is considered as a poset with respect to inclusion. We call $f$ an \emph{lcm-compatible $P$-grading of $\Taylor(I)$} if that there exists a commutative diagram of order-preserving maps
    \[
    \begin{tikzcd}
        \mathcal{P}(\mingens(I)) \arrow[d, "\lcm"] \arrow[rd, "f"]
        & \\
        \mathbb{Z}^d
        & \arrow[l, dashed, "g"]  P.
    \end{tikzcd}
    \]
We also associate to $I$ a directed graph $G_I=(V,E)$, whose vertex and edge sets are
\[
V=\{\sigma \mid \sigma \subseteq \mingens(I)\}
\]
and
\[
E=\{\sigma\to \tau  \mid \tau\subset \sigma \text{ and }|\tau|=|\sigma|-1 \}.
\]

The main objects of discrete Morse theory are defined as follows. 

\begin{definition} \label{def.Morse}
A collection of edges $A\subseteq E$ in $G_I$ is called an \emph{$f$-homogeneous acyclic matching} if the following conditions hold:
    \begin{enumerate}
        \item Each vertex of $G_I$ appears in at most one edge of $A$.
        \item For each directed edge $\sigma\to \tau$ in $A$, we have $f(\sigma)=f(\tau)$.
        \item The directed graph $G_I^A$ --- which is $G_I$ with edges in $A$ being reversed --- is acyclic, i.e., $G_I^A$ does not have any directed cycle.
    \end{enumerate}
\end{definition}

For an $f$-homogeneous acyclic matching $A \subseteq E$ in $G_I$, the subsets of $\mingens(I)$ that are not in any edge of $A$ are called \emph{$A$-critical}. When there is no confusion, we will simply use the term \emph{critical}. The main result of discrete Morse theory essentially says that the critical subsets of $\mingens(I)$ form a free resolution for $S/I$. 

\begin{theorem}[{\cite[Propositions 2.2 and 3.1]{BW02}}] \label{thm:morse-resolutions}
    Let $P$ be a poset, let $f$ be an lcm-compatible $P$-grading of $\Taylor(I)$, and let $A$ be an $f$-homogeneous acyclic matching in $G_I$. Then, $A$ induces a free resolution $\mathcal{F}_A$ of $S/I$, which we call the \emph{Morse resolution} with respect to $A$. Moreover, for each integer $i \ge 0$, a basis of $(\mathcal{F}_A)_i$ can be identified with the collection of critical subsets of $\mingens(I)$ with exactly $i$ elements.   
\end{theorem}

It can be seen that Morse resolutions are contained in the Taylor resolution. Furthermore, the bigger the $f$-homogeneous acyclic matchings are, the smaller the induced Morse resolutions will be. 

\begin{remark}\label{rem:characteristic-free}
   Morse resolutions are cellular~\cite[Proposition 1.2]{BW02} and independent of $\characteristic \Bbbk$.
\end{remark}

We now recall two different constructions for $f$-homogeneous acyclic matchings on $\Taylor(I)$, which result in the generalized Lyubeznik and generalized Barile-Macchia resolutions. 

\begin{theorem}[{\cite[Theorem 3.2]{BW02}}]
    \label{thm:generalized-Lyubeznik}
    Let $P$ be a poset, let $f$ be an lcm-compatible $P$-grading of $\Taylor(I)$, and let $(\succ_p)_{p\in P}$ be a sequence of total orders on $\mingens(I)$.    
    
    For $\sigma=\{ m_1 \succ_{f(\sigma)} \cdots \succ_{f(\sigma)} m_q \}$, we define
   \begin{align*}
        v_L(\sigma)\coloneqq  \sup \big\{ k\in \mathbb{N} \mid \;
        & \exists m \in \mingens(I) \text{ such that } m_k \succ_{f(p)} m \text{ for } k\in [q] \\
        & \text{ and } m \mid \lcm(m_1, \ldots, m_k) \big\}.
   \end{align*}
   
    If $v_L(\sigma)\neq -\infty$, set
   \[m_L(\sigma)\coloneqq \min_{\succ_{f(\sigma)}} \{m\in \mingens(I) ~\big\vert~ m\mid \lcm(m_1,\dots, m_{v_L(\sigma)}) \}.\]
    For each $p\in P$ set
    \[
    A_p\coloneqq \{(\sigma\cup \{m_L(\sigma)\})\to (  \sigma \setminus \{m_L(\sigma)\}) \mid  f(\sigma)=p \text{ and }v_L(\sigma)\neq - \infty \}.
    \]
    Assume that $f(\sigma \setminus \{m_L(\sigma)\}) = f(\sigma\cup \{m_L(\sigma)\})$ holds for each $\sigma \subseteq \mingens(I)$ for which $m_L(\sigma)$ exists. Then $A=\bigcup_{p\in P}A_p$ is an $f$-homogeneous acyclic matching. Hence, it induces a free resolution of $S/I$, called a \emph{generalized Lyubeznik resolution}.
\end{theorem}

The construction in Theorem \ref{thm:generalized-Lyubeznik} generalizes that given by Lyubeznik~\cite{Ly88}. The \emph{Lyubeznik resolution} (with respect to a fixed total order $(\succ)$ on $\mingens(I)$) is exactly the generalized Lyubeznik resolution when $P=\mathbb{Z}^d$, $f=\lcm$, and $(\succ_p) = (\succ)$ for all $p\in P$.

\begin{definition} \label{defi:types}
Fix a total order $(\succ)$ on $\mingens(I)$.
\begin{enumerate} 
    \item Given $\sigma\subseteq \mingens(I)$ and $m\in \mingens(I)$  such that $\lcm(\sigma \cup \{m\})=\lcm(\sigma\setminus \{m\})$, we say that $m$ is a \emph{bridge} of $\sigma$ if $m\in \sigma$.
    \item If $m\succ m'$ where $m,m'\in \mingens(I)$, we say that $m$ \emph{dominates} $m'$.
    \item The \emph{smallest bridge function} is defined to be
    \[
    \sbridge: \mathcal{P}(\mingens(I))\to \mingens(I) \sqcup \{\emptyset\}
    \]
    where $\sbridge(\sigma)$ is the smallest bridge of $\sigma$ (with respect to $(\succ)$) if $\sigma$ has a bridge and $\emptyset$ otherwise.
\end{enumerate}
\end{definition}

\begin{algorithm}\label{algorithm1}
    {\sf Let $A=\emptyset$. Set $\Omega\subseteq \{\text{all  subsets of } \mingens(I) \text{ with cardinality at least } 3\}.$
    \begin{enumerate}[label=(\arabic*)]
        \item Pick a subset $\sigma$ of maximal cardinality in $\Omega$. 
        \item  Set
        \[ \Omega \coloneqq \Omega \setminus \{\sigma, \sigma \setminus \{\sbridge(\sigma)\}\}. \]
        If  $\sbridge(\sigma)\neq \emptyset$, add the directed edge $\sigma \to (\sigma \setminus \{\sbridge(\sigma)\})$ to $A$. \newline
        If $\Omega\neq \emptyset$, return to step (1).
        \item Whenever there exist distinct directed edges $\sigma \to (\sigma \setminus \{\sbridge(\sigma)\})$ and $\sigma'\to ( \sigma' \setminus \{\sbridge(\sigma')\})$ in $A$ such that 
            $$\sigma \setminus \{\sbridge(\sigma)\} = \sigma' \setminus \{\sbridge(\sigma')\},$$
            then 
            \begin{itemize}
                \item if $ \sbridge(\sigma') \succ \sbridge(\sigma)$, remove $\sigma'\to ( \sigma' \setminus \{\sbridge(\sigma')\})$ from $A$,
                \item else remove $\sigma\to ( \sigma \setminus \{\sbridge(\sigma)\})$ from $A$.
            \end{itemize}
        
        \item Return $A$.
    \end{enumerate}}
\end{algorithm}

\begin{theorem}[{\cite[Theorem 5.18]{CK24}}]
    \label{thm:generalized-BM} 
     Let $P$ be a poset, $f$ an lcm-compatible $P$-grading of $\Taylor(I)$, and $(\succ_p)_{p\in P}$ a sequence of total orderings of $\mingens(I)$. For $\sigma \subseteq \mingens(I)$, we set the notation
     \[
     \sbridge(\sigma) \coloneqq \sbridge_{\succ_{f(\sigma)}}(\sigma).
     \]
     Assume $f(\sigma \setminus \{\sbridge(\sigma)\} )=f(\sigma)$ for any subset $\sigma$ of $\mingens(I)$. For each $p\in P$, let $A_p$ be the $f$-homogeneous acyclic matching obtained by applying Algorithm~\ref{algorithm1} to the set $f^{-1}(p)$ imposed with the total ordering $(\succ_p)$. Then $A=\bigcup_{p\in P}A_p$ is an $f$-homogeneous acyclic matching. Hence, it induces a free resolution of $S/I$, called a \emph{generalized Barile-Macchia resolution}.
\end{theorem}

Similar to what we have seen with Lyubeznik resolutions, a \emph{Barile-Macchia resolution} (with respect to a fixed total order $(\succ)$ on $\mingens(I)$) is exactly the generalized Barile-Macchia resolution when $P=\mathbb{Z}^d$, $f=\lcm$, and $(\succ_p)=(\succ)$ for any $p\in P$.

Observe that, while Barile-Macchia algorithm matches subsets of $\mingens(I)$ with a priority based on their cardinality, the matchings that induce generalized Lyubeznik resolutions are established regardless of the order they are matched. Thus, we can assume that generalized Lyubeznik resolutions are induced using the same rules as generalized Barile-Macchia resolutions in Theorem~\ref{thm:generalized-BM}, where a modification of Algorithm~\ref{algorithm1} is used. The change is simple: replace $\sbridge$ with $m_L$.
The two constructions are similar in the sense that they have almost the same inputs, and that they coincide in some important cases. 

\begin{lemma}\label{lem:Lyu=BM}
    Let $P$ be a poset, $f$ an lcm-compatible $P$-grading of $\Taylor(I)$, and $(\succ_p)_{p\in P}$ a sequence of total orderings of $\mingens(I)$. Assume that for each $\sigma \subseteq \mingens(I)$ where $m_L(\sigma)$ exists, we have $f(\sigma \setminus \{m_L(\sigma)\} )=f(\sigma\cup \{m_L(\sigma)\})$. Assume that the corresponding generalized Lyubeznik resolution of $S/I$ is minimal and $m_L(\sigma)=\sbridge(\sigma)$ for any $\sigma\subseteq \mingens(I)$ where $m_L(\sigma)$ exists and is in $\sigma$. Then, the corresponding generalized Barile-Macchia resolution of $S/I$ is isomorphic to the generalized Lyubeznik resolution and, in particular, is minimal.
\end{lemma}

\begin{proof}
    Let $A_L = \bigcup_{p\in P} (A_L)_p$ denote the homogeneous acyclic matching that induces the generalized Lyubeznik resolution $\mathcal{F}_L$ in this case. Then for each $p\in P$, we have
    \begin{align*}
        (A_L)_p&=\{(\sigma\cup \{m_L(\sigma)\})\to (  \sigma \setminus \{m_L(\sigma)\}) \mid  f(\sigma)=p \text{ and }v_L(\sigma)\neq - \infty \}\\
        &=\{\sigma\to (\sigma \setminus \{m_L(\sigma)\}) \mid  f(\sigma)=p, \ m_L(\sigma) \text{ exists and is in } \sigma  \}\\
        &=\{\sigma\to (  \sigma \setminus \{\sbridge(\sigma)\}) \mid  f(\sigma)=p, \ m_L(\sigma) \text{ exists and is in } \sigma  \}.
    \end{align*}
    Recalling the discussion before this result, we can assume $\sigma$ here is chosen based on cardinality, and thus coincides with how the Barile-Macchia algorithm works. Let $A_{BM}$ denote the homogeneous acyclic matching that induces the generalized Barile-Macchia resolution $\mathcal{F}_{BM}$ in this case. By Step~(3) of the algorithm, $A_{BM}$ is exactly $A_L$ after replacing and adding some directed edges. Because $\mathcal{F}_L$ already induces the minimal resolution by the hypotheses, adding edges is impossible by Theorem~\ref{thm:morse-resolutions}. Thus $A_{BM}$ is exactly $A_L$ after (potentially) replacing some edges. Again by Theorem~\ref{thm:morse-resolutions}, $\rank (\mathcal{F}_{BM})_i=\rank (\mathcal{F}_{L})_i$ for any index $i$. Thus $\mathcal{F}_{BM}$ is also minimal, and isomorphic to $\mathcal{F}_{L}$ as a consequence.
\end{proof}

The hypotheses in Lemma \ref{lem:Lyu=BM}, while seem restrictive, hold for both of the only classes of ideals which are known to have minimal generalized Lyubeznik resolutions (see Theorems~\ref{thm:generic-BM} and~\ref{thm:shellable-BM}).


\section{Generic Monomial Ideals and Ideals with Linear Quotients} \label{sec.gen_shellable}

In this section, we prove results parallel to those of Batzies and Welker~\cite{BW02} for generic monomial ideals and for monomial ideals with linear quotients. As in Section~\ref{sec.Morse}, $S = \Bbbk[x_1, \dots, x_d]$ denotes a polynomial ring over a field $\Bbbk$ and $I \subseteq S$ is a monomial ideal.

We start by recalling the definition of generic monomial ideals, following~\cite{MSY00}. For a monomial $m$, let $\text{ord}_i(m)$ be the highest power of $x_i$ that divides $m$, for $i=1, \dots, d$. For a monomial ideal $I$, set $\text{ord}_i(I)\coloneqq \min \{\text{ord}_i(m) \mid m\in \mingens(I)\}$, for $i=1, \dots, d$. A monomial ideal $I$ is called \emph{generic} if whenever there exist two different monomials $m,m'\in \mingens(I)$ with $\text{ord}_i(m) = \text{ord}_i(m') > \text{ord}_i(I)$, for some $1 \le i \le d$, then there exists a third monomial $m''\in \mingens(I)$ which divides $\lcm(m,m')$ and for any $1 \le j \le d$, 
$$\max\{\text{ord}_j(m), \text{ord}_j(m')\} > \text{ord}_j(m'') \text{ if and only if } \max\{\text{ord}_j(m), \text{ord}_j(m')\}> \text{ord}_j(I).$$ 

\begin{theorem}\label{thm:generic-BM}
    Let $I$ be a generic monomial ideal. Then $S/I$ has a minimal generalized Barile-Macchia resolution.
\end{theorem}
\begin{proof}
    We recall the generalized Lyubeznik resolution that minimally resolved $S/I$~\cite[Proposition 4.1]{BW02}.  Set $P=\mathbb{N}^d$ with the natural partial order and $f=\lcm$. Let $p\in P$ be a monomial such that there exists $\sigma\subseteq \mingens(I)$ with $\lcm(\sigma)=p$. Let $\sigma_1,\dots, \sigma_k$ be the minimal subsets of $\mingens(I)$ such that $\lcm(\sigma_1)=\cdots = \lcm(\sigma_k)=p$. Then we define a total order $(\succ_p)$ on $\mingens(I)$ so that elements of $\Sigma_p=\sigma_1\cup \cdots \cup \sigma_k$ are the biggest. By~\cite[Proposition 4.1]{BW02}, the generalized Lyubeznik resolution of $S/I$ using these ingredients is minimal. By Lemma~\ref{lem:Lyu=BM}, it suffices to show that for any $\sigma\subseteq \mingens(I)$ such that $m_L(\sigma)$ exists and is in $\sigma$, we have
    \begin{equation*}\label{equa:sb=mL-generic}
        \sbridge(\sigma)=m_L(\sigma).
    \end{equation*}
    Fix $p=\lcm(\sigma)$. We claim that $m_L(\sigma)\notin \Sigma_p$. If $\sigma$ contains exactly one of the $\sigma_i$, then this follows immediately from the minimality hypothesis of $\sigma_i$. Now we can assume, without loss of generality that $\sigma$ contains $\sigma_1$ and $\sigma_2$. Then there must be a variable $x_r$ and two monomials $m\in \sigma_1$ and $m'\in \sigma_2$ such that $\text{ord}_r(m) = \text{ord}_r(m') > \text{ord}_r(I)$. By genericity, there exists a monomial $m''\in \mingens(I)$ that divides $\lcm(m,m')$ such that for any index $j$, if $\text{ord}_j(m'') > \text{ord}_j(I)$, then $\max\{\text{ord}_j(m), \text{ord}_j(m')\} > \text{ord}_j(m'')$. In other words, $m''\notin \Sigma_p$. By definition, we have $m''\succ_p m_L(\sigma)$, and thus $m_L(\sigma)\notin \Sigma_p$, as claimed. 
    
    Back to proving (\ref{equa:sb=mL-generic}), we first have $m_L(\sigma) \succeq_p \sbridge(\sigma)$ by definition. In particular, this implies that $\sbridge(\sigma)\notin \Sigma_p$. Couple this with the facts that $\sbridge(\sigma)\in \sigma$ and $\lcm(\sigma)=\lcm(\sigma_1)$, we must have
    \[
    \sbridge(\sigma) \mid \lcm(\sigma_1) \mid \lcm\left(\{n\in \mingens(I)  \mid n\succ \sbridge(\sigma)\}\right).
    \]
    By definition 
    $m_L(\sigma) \preceq_p \sbridge(\sigma)$. This concludes the proof.
\end{proof}

We turn our attention to monomial ideals with linear quotients. A monomial ideal $I$ is said to have \emph{linear quotients} if there exists an total order $(\sqsupset)$ on $\mingens(I)$ such that for any $m,m'\in \mingens(I)$ with $m\sqsupset m'$, there exists an $m''\in \mingens(I)$ such that $m\sqsupset m''$ and $\lcm(m,m'')=mx_{g(m,m'')}$ divides $\lcm(m,m')$ for some index $g(m,m'')$. 

\begin{theorem}\label{thm:shellable-BM}
    Let $I$ be a monomial ideal with linear quotients. Then $S/I$ has a minimal generalized Barile-Macchia resolution.
\end{theorem}

\begin{proof}
    We recall the generalized Lyubeznik resolution that minimally resolved $S/I$~\cite[Proposition 4.3]{BW02}. Let $(\sqsupset)$ be a total order on the minimal generators of $I$ as in the definition of linear quotients. Set
\begin{align*}
    \lcm(I)&\coloneqq \{\lcm(\sigma) \mid \sigma \subseteq \mingens(I) \},\\
    M_\alpha&\coloneqq \{m\in \mingens(I) \mid \exists \sigma \subseteq \mingens(I) \text{ such that } \lcm\sigma = \alpha \text{ and } m=\max_{\sqsupset} \sigma \},
\end{align*}
for any monomial $\alpha$. Let
\[
P\coloneqq \{(\alpha,m) ~\big|~ \alpha\in \lcm(I), m\in M_\alpha\}
\]
be a poset with partial order given by 
\[(\alpha,m)\geq (\alpha',m') \Longleftrightarrow (\alpha>\alpha') 
\text{ or } 
(\alpha=\alpha' \text{ and } m\sqsupset m').\]  
For $\sigma\subseteq \mingens(I)$, define \[f(\sigma)\coloneqq (\lcm\sigma, \max_\sqsupset (\sigma)).\]
For each $m\in\mingens(I)$, set
\[
J_m\coloneqq \{j\in [d] \mid \exists n_j^m\in \mingens(I) \text{ such that } n_j^m\sqsubset m \text{ and } \lcm(n_j^m,m)=x_jm \}.
\]
For each $j\in J_m$, fix a monomial $n_j^m$. Now for each $p=(\alpha,m) \in P$, we define a total order $(\succ_p)$ on $\mingens(I)$ by setting
\begin{itemize}
    \item $N_m\coloneqq \{n_j^m ~\big|~ j\in J_m\}$,
    \item $\mingens(I)\setminus N_m \succ_p N_m$,
    \item for $n_j^m, n_{j'}^m\in N_m$, we have $n_j^m\succ_p n_{j'}^m \Longleftrightarrow j>j'$,
    \item $\succ_p \mid_{\mingens(I)\setminus N_m}=\sqsupset_p \mid_{\mingens(I)\setminus N_m}$.
\end{itemize}
By~\cite[Proposition 4.3]{BW02}, the generalized Lyubeznik resolution of $S/I$ using these ingredients is minimal. In fact, they explicitly described the $f$-homogeneous acyclic matching  $A = \bigcup A_p$ in this case: 
    \[
    A_p\coloneqq \{(\sigma \cup \{m_L(\sigma)\})\to (\sigma \setminus \{m_L(\sigma)\})  \mid f(\sigma)=p \text{ and } (\mingens(I) \setminus N_m ) \cap \sigma \supsetneq \{m\} \}.
    \]
    Fix $p=(\alpha,m)$. By Lemma~\ref{lem:Lyu=BM}, it suffices to show that for any $\sigma\subseteq \mingens(I)$ where $f(\sigma)=p$ and $\big({\mingens}(I) \setminus (N_m \cup \{m\}) \big) \cap \sigma \neq \emptyset$, we have
    \begin{equation}\label{equa:sb=mL}
        \sbridge(\sigma \cup \{m_L(\sigma)\}) = m_L(\sigma).
    \end{equation}
    We claim that $m_L(\sigma) \in N_m$. Indeed, let $m'\in \big({\mingens}(I) \setminus (N_m \cup \{m\}) \big) \cap \sigma$. We have $m=\max_\sqsupset (\sigma) \sqsupset m'$. Hence by definition, there exists $m''\in \mingens(I)$ such that $m\sqsupset m''$ and $\lcm(m,m'')=mx_j$ divides $\lcm(m,m')$ for some index $j$. By our construction, $j\in J_m$ and we can assume that $m''=n_j^m$. We remark that by our total order, $m'\sqsupset m''$, and $m''$ divides $\lcm(m,m')$. By definition, $m''\succ_p m_L(\sigma)$. Since $m''=n_j^m\in N_m$, so is $m_L(\sigma)$. Thus the claim holds. Due to this, we have $\big({\mingens}(I) \setminus (N_m \cup \{m\}) \big) \cap (\sigma \cup \{m_L(\sigma)\}) \neq \emptyset$. Therefore we can assume that $m_L(\sigma)\in \sigma$ since we already know that $f(\sigma \cup \{m_L(\sigma)\}) = f(\sigma)=p$. By definition, we have 
    \begin{equation}\label{equa:ml-sb-1}
        m_L(\sigma) \succeq_p \sbridge(\sigma).
    \end{equation}
    In particular, this implies that $\sbridge(\sigma) \in N_m$, i.e., there exists an index $k\in J_m$ such that $\sbridge(\sigma)=n_k^m$. Since $\lcm(n_k^m,m)=x_km$ and  $m\in \sigma$, we have $n_k^m$ is a bridge of $\sigma$ if and only if $x_k^{a+1}$ divides $\lcm(\sigma \setminus \{n_k^m\})$, where  $x_k^a$ is the highest power of $x_k$ that divides $m$. In particular, this implies that there exists a monomial $m'\in \sigma$ such that $n_k^m$ divides $\lcm(m,m')$. Then $x_km=\lcm(n_k^m,m)$ divides $\lcm(m,m')$, which implies that $m'\notin N_m$. Thus $m'\sqsupset n_k^m$ and $n_k^m$ divides $\lcm(m,m')$. By definition, we have
    \begin{equation}\label{equa:ml-sb-2}
        m_L(\sigma) \preceq_p n_k^m= \sbridge(\sigma).
    \end{equation}
    Combining (\ref{equa:ml-sb-1}) and (\ref{equa:ml-sb-2}), we obtain (\ref{equa:sb=mL}), as desired.    
\end{proof}

Focusing on edge ideals, we recall the following equivalent conditions:
\begin{enumerate}[label=(\roman*)]
    \item $G$ is a co-chordal graph, i.e., the complement graph of $G$ is chordal.
    \item $I(G)$ has a linear resolution.
    \item $I(G)$ has linear quotients.
\end{enumerate}
Here (i)$\Longleftrightarrow$(ii) is the celebrated Fr\"oberg theorem~\cite{Fro90}, (ii)$\Longleftrightarrow$(iii) is proved by Herzog-Hibi-Zheng~\cite[Theorem 3.2]{HHZ03}. 
We thus obtain the following corollary.

\begin{corollary}\label{cor:co-chordal-BM}
    Let $G$ be a co-chordal graph. Then $I(G)$ has a minimal generalized Barile-Macchia resolution.
\end{corollary}

Inspired by the results of~\cite{BW02} and Theorems \ref{thm:generic-BM} and \ref{thm:shellable-BM}, we raise the following conjecture.

\begin{conjecture}
    If a monomial ideal $I$ has a minimal generalized Lyubeznik resolution, then $I$ also has a minimal generalized Barile-Macchia resolution.
\end{conjecture}


\section{Edge Ideals of Graphs} \label{sec.edgeIdeal}

This section focuses on edge ideals of graphs. We shall identify classes of graphs whose edge ideals have minimal generalized Barile-Macchia resolutions. 
Throughout this section, $G$ denotes a \emph{simple} graph (i.e., $G$ contains no loops nor multiple edges) with vertex set $V(G) = \{x_1, \dots, x_d\}$ and edge set $E(G)$. As before, let $S = \Bbbk[x_1, \dots, x_d]$ be a polynomial ring over a field $\Bbbk$. Following, for example, \cite{villarreal1990cohen}, the \emph{edge ideal} of $G$ is defined by
\[I(G) = (x_ix_j \mid\{x_i,x_j\} \in E(G)) \subseteq S.\]

It was seen in~\cite{CHM24-first} that, while the property of having a minimal Lyubeznik resolution for edge ideals of graphs is quite well understood, virtually nothing is known about the property of having a minimal Barile-Macchia or a minimal generalized Barile-Macchia resolutions. On the other hand, computational experiments show that Barile-Macchia and generalized Barile-Macchia resolutions are effective in generating the minimal free resolution of edge ideals.

We start with the following statement that is verified by \texttt{Macaulay2}~\cite{M2} computations.

\begin{theorem} \label{thm:BM-graphs}
    For any graph $G$ over at most 6 vertices, $S/I(G)$ has a minimal Barile-Macchia resolution. On the other hand, if $G$ is the following graph on $7$ vertices
    \begin{center}
		\smallestcounterexamplesymb
	\end{center}
    then $S/I(G)$ does not have a minimal Barile-Macchia resolution.
\end{theorem}

The graph given in Theorem \ref{thm:BM-graphs} is the smallest example whose edge ideal cannot be minimally resolved by Barile-Macchia resolutions, in terms of both the number of vertices and the number of edges. 
For generalized Barile-Macchia resolutions, we can do slightly better. To facilitate this, we will introduce an algorithm to find a minimal generalized Barile-Macchia resolution for $S/I(G)$, if it exists, when $G$ is a graph of at most 10 vertices. 

Observe that, by~\cite[Theorem 4.1]{Kat04}, the graded Betti numbers of $I(G)$ are characteristic-independent in thit case. Because Morse resolutions in general are also characteristic-independent (Remark \ref{rem:characteristic-free}), to verify if a Morse resolution $\mathcal{F}_A$ of $S/I(G)$ is minimal we only need to do so in characteristic 2. Observe further that, for a given monomial $m$, comparing $\beta_{i,m}(S/I(G))$ to the number of critical subsets of $\mingens(I(G))$ with lcm $m$ and cardinality $i$, for all $i \in \NN$, is the same as comparing $\sum_{i=0}^{\pd S/I(G)} \beta_{i,m}(S/I(G))$ to the number of critical subsets of $\mingens(I(G))$ with lcm $m$.



In the following algorithm, if the output is ``True'', then $I(G)$ has a minimal generalized Barile-Macchia resolution, and if the output is ``False'', then it is unknown whether a minimal generalized Barile-Macchia resolution for $I(G)$ exists.

\begin{algorithm} \label{algo:aryaman}
    {\sf
    Input: a graph $G$ with at most 10 vertices. Set $S=\mathbb{Z}/2\mathbb{Z}[V(G)]$ and $\mathcal{V}\coloneqq \mathcal{P}(V(G))$. 
    \begin{enumerate}[label=(\arabic*)]
        \item If $G$ is co-chordal, then return True (see Corollary~\ref{cor:co-chordal-BM}). 
        \item If $\mathcal{V}\neq \emptyset$, pick $V$ in $\mathcal{V}$, and let $\Omega$ be the set of all total orders on $\mingens(I(G_V))$, where $G_V$ denotes the induced subgraph of $G$ with vertices in $V$, and set
        \[
        \mathcal{V}\coloneqq \mathcal{V}\setminus \{V\}.
        \]
        Else return True.
        \item Set $m=\prod_{x\in V}x$. Compute 
            \[
            a\coloneqq \sum_{i=1}^{\pd S/I(G)} \beta_{i,m}(S/I(G)).
            \]
        \item If $\Omega\neq \emptyset$, pick $(\succ)$ in $\Omega$, set
        \[
        \Omega\coloneqq \Omega \setminus \{\succ\}
        \]
        and compute
        \[
            b\coloneqq \#\{ \sigma \subseteq E(G)  ~\big|~ \sigma \text { is Barile-Macchia-critical with respect to } (\succ) \text{ and } \lcm \sigma=m \}.
        \]
        Else return False.
        \item If $a=b$, then go to Step (2). Otherwise, go to Step (4).
    \end{enumerate}
    }
\end{algorithm}

We remark that performing Steps (2)--(5) of the algorithm for a graph $G$ and a subset of vertices $V\subseteq V(G)$ is the same as doing so for the graph $G_V$ and its vertex set $V=V(G_V)$. This is thanks to the Restriction Lemma~\cite[Lemma 2.14]{CHM24-first} (the proof of~\cite[Lemma 2.14]{CHM24-first} holds for Morse resolutions in general). Therefore, in practice, we applied the algorithm to graphs with smaller number of vertices first, and so in Step (2), we only considered $V=V(G)$. This remarkably cut down the running time of the algorithm.

By implementing Algorithm \ref{algo:aryaman} in this fashion, we arrive at the following result.

\begin{theorem} \label{thm:generalized-BM-graphs}
    The edge ideal $I(G)$ has a minimal generalized Barile-Macchia resolution for any graph $G$ with 8 vertices or less.
\end{theorem}
\begin{proof}
    The database~\cite{GBM} contains the total order on $\mingens(I(G))$ for any such $G$.
\end{proof}

While Algorithm \ref{algo:aryaman}, in theory, works for any graph with at most 10 vertices (or, in fact, any graph whose graded Betti numbers are characteristic-independent), our supercomputers cannot keep up with graphs over 9 vertices. At least for 220000 out of 261080 graphs with 9 vertices, their edge ideals have minimal generalized Barile-Macchia resolutions (also available in the database~\cite{GBM}). The remaining graphs on $9$ vertices required computational speeds that are not available to us and consumed time beyond our capacity. However, we expect that the edge ideals of all graphs on at most 10 vertices have minimal generalized Barile-Macchia resolutions.

On the other hand, since Morse resolutions do not depend on the characteristic, the characteristic dependence of the graded Betti numbers of $I(G)$ may determine if $I(G)$ does not have a minimal generalized Barile-Macchia resolution. The smallest graphs, whose edge ideals have characteristic-dependent graded Betti numbers, contain 11 vertices, and there are four of them~\cite[Appendix A]{Kat04}. The following question is certainly of interest.

\begin{question}
    Characterize the graphs $G$ for which $I(G)$ has a minimal Barile-Macchia or a minimal generalized Barile-Macchia resolution.
\end{question}


In general, showing that a (generalized) Barile-Macchia is minimal is typically very difficult. However, there is a sufficient condition that is more tractable via the notion of \emph{bridge-friendly} monomial ideals, which we shall now recall from~\cite{CK24}. 

\begin{definition} \label{defi:types-second-part}
Let $I \subseteq S$ be a monomial ideal and fix a total order $(\succ)$ on $\mingens(I)$ (see also Definition~\ref{defi:types}).
\begin{enumerate} 
    \item Given $\sigma\subseteq \mingens(I)$ and $m\in \mingens(I)$  such that $\lcm(\sigma \cup \{m\})=\lcm(\sigma\setminus \{m\})$, we say that $m$ is a \emph{gap} of $\sigma$ if $m\notin \sigma$.
    \item A monomial $m\in \mingens(I)$ is called a \emph{true gap} of $\sigma\subseteq \mingens(I)$ if 
        \begin{enumerate}
            \item[(a)]  it is a gap of $\sigma$, and 
            \item[(b)]  the set $\sigma \cup \{m\}$ has no new bridges dominated by $m$. In other words, if $m'$ is a bridge of $\sigma \cup \{m\}$ and $m\succ m'$, then $m'$ is a bridge of $\sigma$.
        \end{enumerate}
    Equivalently, $m$ is not a true gap of $\sigma$ either if $m$ is not a gap of $\sigma$ or if there exists $m'\prec m$ such that $m'$ is a bridge of $\sigma \cup \{m\}$ but not one of $\sigma$.
    \item A subset $\sigma\subseteq \mingens(I)$ is called \emph{potentially-type-2} if it has a bridge not dominating any of its true gaps, and \emph{type-1} if it has a true gap not dominating any of its bridges. Moreover, $\sigma$ is called \emph{type-2} if it is potentially-type-2 and whenever there exists another potentially-type-2 $\sigma'$ such that 
    \begin{equation*} 
    	\sigma' \setminus \{\sbridge(\sigma')\}=\sigma \setminus \{\sbridge(\sigma)\},
    \end{equation*}
    we have $\sbridge(\sigma')\succ \sbridge(\sigma)$.
\end{enumerate}
\end{definition}

\begin{definition}[{\cite[Definition 2.27]{CK24}}]
    A monomial ideal $I \subseteq S$ is \emph{bridge-friendly} if there exists a total order $(\succ)$ on $\mingens(I)$ such that all potentially-type-2 subsets of $\mingens(I)$ are type-2.    
\end{definition}

A bridge-friendly monomial ideal has a minimal Barile-Macchia resolution by~\cite[Theorem 2.29]{CK24}. In the rest of this section, we focus on connected unicyclic graphs and characterize those whose edge ideals are bridge-friendly. 

The following statement identifies small graphs that are not bridge-friendly and serve as ``forbidden structures'' for this property.

\begin{proposition}\label{prop:non-bridgefriendly-graphs}
    Let $G$ be one of the following graphs:
    \begin{enumerate}
        \item The net graph \sunletsymb{3}.
        \item The 5-sunlet graph \sunletsymb{5}.
        \item The 123-trimethylcyclohexane graph \sixcycleconsecutivesymb.
        \item The 135-trimethylcyclohexane graph  \sixcyclealternatesymb.
    \end{enumerate}
    
    Then $I(G)$ is not bridge-friendly.
\end{proposition}

\begin{proof}
    Verified with \texttt{SageMath} computations.
\end{proof}

Recall that a connected graph is called a \emph{tree} if it has no cycles. A connected unicyclic graph consists of a cycle, say $C_n$, that is joined at its vertices with at most $n$ trees. It is easy to see that all the graphs in Proposition \ref{prop:non-bridgefriendly-graphs} are unicyclic.

For a tree $T$, consider the following particular total order ($\succ$) on the edge set $E(T)$ of $T$. Fix vertex $x_0$ in $T$, and view $T$ as a \emph{rooted} tree with root $x_0$. Each vertex $v \in V(T)$ determines a unique path from $v$ to $x_0$. For $i \in \NN$, let 
\begin{equation*} 
	V_i \coloneqq \left\{v \in V(T) \mid \dist_T(v,x_0) =i\right\}
\end{equation*} 
be the set of vertices in $T$ whose distance to $x_0$ is $i$. Obviously, $V(T)=\bigcup_{i\in \mathbb{Z}_{\geq 0}} V_i$. Let $c_i = \card{V_i}$, for $i \in \ZZ_{\ge 0}$. We shall consider a  specific labeling for the vertices in $T$ given by writing 
\begin{equation*} 
	V_i = \left\{x_{i,j} ~\middle|~ 1 \le j \le c_i\right\},
\end{equation*}
with the convention that $x_{0,1} = x_0$. 
With respect to this particular labeling of the vertices in $T$, define the following total order $(\succ)$ on $E(T)$:
    \[
    x_{i,j}x_{i+1,k}\succ x_{i',j'}x_{i'+1,k'} \text{ if } \begin{sqcases}
        i<i'; \text{ or}\\
        i=i' \text{ and } j<j';\text{ or} \\
        i=i', j=j' \text{ and } k<k'. 
    \end{sqcases}
    \]

Our next main result is stated as follows.

\begin{theorem}\label{thm:bridgefriendly-unicyclic}
    Let $G$ be a connected unicyclic graph. Then, $I(G)$ is bridge-friendly if and only if either
    \begin{enumerate} 
    \item $G$ contains a $C_3$ or a $C_5$ with one vertex of degree 2; or 
    \item $G$ contains a $C_6$ with two opposite vertices of degree 2.
    \end{enumerate}
\end{theorem}

\begin{proof}
    By~\cite[Proposition 4.2]{CHM24-first}, if $I(G)$ is bridge-friendly, then the only cycle in $G$ must be one of $C_3, C_5$, or $C_6$. Together with the structure of forbidden unicyclic graphs given in Proposition~\ref{prop:non-bridgefriendly-graphs}, it is follows that this unique cycle in $G$ has to be either a $C_3$ with a vertex of degree 2, or a $C_5$ with a vertex of degree 2, or a $C_6$ with two opposite vertices of degree 2. This establishes the ``only if" part. For the ``if" part, let $G$ be a unicyclic graph of the described form. We will show that $I(G)$ is bridge-friendly. 
    
    If $G$ contains a $C_3$ with a vertex of degree 2, then the conclusion follows from~\cite[Theorem 4.8]{CHM24-first}, where bridge-friendly edge ideals of chordal graphs are fully characterized.
    Assume that $G$ contains a $C_5$ or a $C_6$. By contradiction, suppose that $I(G)$ is not bridge-friendly.  
    
    By~\cite[Proposition 4.1]{CHM24-first}, for any total order $(\succ)$ on $E(G)$, there exists a collection of edges $\tau \subseteq E(G)$ and edges $m_1\succ m_2 \succ m_3$ in $E(G)$ such that if we set $m_3=yz$, then no other edge in $\tau$ contains $y$ or $z$; $m_1,m_2$ are true gaps of $\tau$, $m_1$ does not dominate any true gap of $\tau\cup \{m_1\}$, and $\sbridge(\tau\cup \{m_1\})=m_1$. 
    
    Consider the case where $G$ contains a $C_5$, whose edges are $\{x_1x_2,x_2x_3,x_3x_4,x_4x_5,x_1x_5\}$, and assume that $x_1$ is a a vertex of degree 2 in $G$. Particularly, the neighbors of $x_1$ in $G$ are exactly $x_2$ and $x_5$. For $i = 2, \dots, 5$, we denote the tree attached to the vertex $x_i$ of the unique $C_5$ in $G$ by $T_i$, and view $T_i$ as a rooted tree with root $x_i$. For each rooted tree $T_i$, $i =2, \dots, 5$, let $(\succ)$ denote the total order on its edges as described above. We extend these into a total order $(\succ)$ on $E(G)$ as follows:
    \[
    x_1x_5\succ x_1x_2\succ x_2x_3\succ x_3x_4\succ x_4x_5\succ E(T_2)\succ \cdots \succ E(T_5).
    \]
    
    The edge $m_3$ satisfies the property that there are two other edges (with respect to $(\succ)$) dominating it and containing the its two ends. The only such possibility is $m_3=x_4x_5$. This forces $m_1=x_1x_5$, $m_2=x_3x_4$, and in particular, implies that $\tau$ does not have any edge, other than $m_3$, that contains $x_4$ or $x_5$. Since $m_1=\sbridge(\tau \cup \{m_1\})$, we have $x_1x_2\in \tau$, and $x_1x_2$ is not a bridge of $\tau\cup \{m_1\}$. Hence $\tau$ does not have any edge, other than $x_1x_2$, that contains $x_2$, including $x_2x_3$. Next, the fact that $m_2=\sbridge(\tau \cup \{m_2\})$ implies that no edge in $\tau\cup \{m_2\}$ containing $x_3$ is a bridge of $\tau\cup \{m_2\}$. Since $m_2$ and $x_2x_3$ share the vertex $x_3$, no edge in $\tau\cup \{m_1, x_2x_3\}$ containing $x_3$ is a bridge of $\tau\cup \{m_1, x_2x_3\}$, either. Combining this with the above result that $\tau$ does not have any edge, other than $x_1x_2$, that contains $x_2$, the set $\tau \cup \{m_1, x_2x_3\}$ does not have any bridge smaller than $x_2x_3$ itself. By definition, $x_2x_3$ is a true gap of $\tau\cup \{m_1\}$, and hence $m_1$ dominates a true gap in $\tau\cup \{m_1\}$, a contradiction.

    Finally, suppose that $G$ contains a $C_6$, whose edges are $\{x_1x_2,x_2x_3,x_3x_4,x_4x_5,x_5x_6,x_1x_6\}$, and assume that $x_1$ and $x_4$ are of degree 2. As before, for $i \not= 1, 4$, let $T_i$ denoted the rooted tree attached to the vertex $x_i$ on the unique $C_6$ in $T$. We extend the total order $(\succ)$ on  $E(T_i)$'s to that on $E(G)$ in the same manner as before, namely,
    \[
    x_1x_6\succ x_1x_2\succ x_2x_3\succ x_3x_4\succ x_4x_5\succ x_5x_6\succ E(T_1)\succ \cdots \succ E(T_6).
    \]
    By similar arguments, the only possibility for $m_3$ is $m_3=x_5x_6$. This forces $m_1=x_1x_6$, $m_2=x_4x_5$, and in particular, implies that $\tau$ does not have any edge, other than $m_3$, that contains $x_5$ or $x_6$. Since $m_1$ and $m_2$ are both gaps of $\tau$, and $x_1$ and $x_4$ are both of degree 2, we must have $x_1x_2,x_3x_4\in \tau$. The fact that $m_1=\sbridge(\tau\cup \{m_1\})$, in particular, implies that $x_2x_3\notin \tau$, and that $\tau\cup \{m_1\}$ has exactly one bridge, namely $m_1$ itself. Since $\tau\cup \{m_1\}$ already has $x_1x_2$ and $x_3x_4$, the set of all bridges of $\tau\cup \{m_1, x_2x_3\}$ is a subset of $\{m_1,x_1x_2,x_3x_4\}$. Since $x_4$ is of degree 2, $x_3x_4$ is not a bridge of $\tau\cup \{m_1, x_2x_3\}$. In summary, $\tau\cup \{m_1, x_2x_3\}$ does not have a bridge dominated by  $x_2x_3$ itself. By definition, $x_2x_3$ is a true gap of $\tau\cup \{m_1\}$, and hence $m_1$ dominates a true gap in $\tau\cup \{m_1\}$, a contradiction.
\end{proof}


\section{Edge Ideals of Rooted Hypertrees}\label{sec.hypertree}

In this section, we study squarefree monomial ideals in more general contexts, i.e., those that are not necessarily generated in degree 2. These are viewed as edge ideals of hypergraphs. Our results show that edge ideals of rooted hypertrees possess minimal Barile-Macchia resolutions. Particularly, our results generalize and extend many known results on edge ideals of trees and rooted trees.

A \emph{hypergraph} $\H = (V,\E)$ consists of a vertex set $V = \{x_1, \dots, x_d\}$ and an edge set $\E$, whose elements are subsets of $V$. We restrict our attention to \emph{simple} hypergraphs; that is, when there is no nontrivial containment between the edges in $\E$. A simple hypergraph is also referred to as a \emph{Sperner system}. A simple graph is a simple hypergraph in which each edge has cardinality 2.
As before, let $S = \Bbbk[x_1, \dots, x_d]$ be a polynomial ring over a field $\Bbbk$. The \emph{edge ideal} of a hypergraph $\H$ is constructed in a similar fashion as that of a graph (see \cite{HVT2008}). Particularly,
\[
I(\mathcal{H})\coloneqq \left\langle \prod_{x\in e} x ~\middle\vert~ e \in \E\right\rangle \subseteq S.
\]

A hypergraph is equipped with various graphical structures. One such structure comes from the concept of host graphs. This concept has motivations from optimization theory; see, for instance~\cite{Ber90, BDCV98}. Specifically, a \emph{host graph} of a hypergraph $\H = (V,\E)$ is a graph $H$ over the same vertex set $V$ such that, for each edge $e \in \E$, the induced subgraph of $H$ over the vertices in $e$ is a connected graph. Note that the complete graph is always a host graph of any hypergraph over the same vertex set. Also, a given hypergraph may have several host graphs.

\begin{figure}[h]
        \begin{center}
            \includegraphics{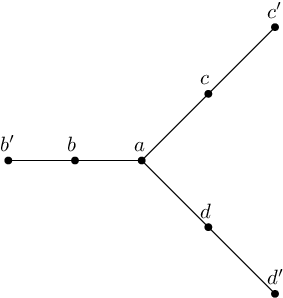}
    	\end{center}
        \caption{A host graph of $\mathcal{H}$. \label{fig.1}}
\end{figure}
\begin{example} \label{ex:hypertree-non-rooted}
Consider the hypergraph $\H$ with edges
$$\left\{ \{a,b,b'\}, \{a,c,c'\}, \{a,d,d'\}, \{a,b,c\}, \{a,c,d\}, \{a,b,d\}\right\}.$$
It is easy to see that Figure \ref{fig.1} depicts a host graph of $\H$.
\end{example}

\begin{definition}
    A hypergraph is called a \emph{hypertree} (respectively, \emph{hyperpath}) if it has a host graph that is a tree (respectively, path).
\end{definition}

This concept of hypertrees has been studied in graph theory and has found many applications in optimization theory (cf.~\cite{Ber90, BDCV98}). We now consider a particular class of hypertrees, whose edge ideals encompass many important classes of edge ideals that have been much studied in the literature, for example, edge ideals of trees~\cite{BM20,CK24} and path ideals of rooted trees~\cite{BHO11}.

\begin{definition} \label{def.rootedtree}
    A hypertree $\H=(V,\E)$ is called \emph{rooted} at a vertex $x \in V$ if there is a host graph $H$ of $\H$, that is a tree, with the property that each edge in $\H$ consists of vertices of different distances from $x$ in $H$. In this case, $x$ is called the \emph{root} of the hypertree $\H$.
\end{definition}

The main result of this section is stated as follows.

\begin{theorem}\label{thm:rooted-hypertree-friendly}
    Let $\H$ be a rooted hypertree. Then, $I(\H)$ is bridge-friendly. In particular, it has a minimal Barile-Macchia resolution.
\end{theorem}

\begin{proof}
    Let $\mathcal{H}=(V,\mathcal{E})$ be a rooted hypertree with root $x_1^{(0)}$, and let $H$ be its host graph, as in Definition \ref{def.rootedtree}. Since $H$ is a tree, we can write its vertices as
    \[
    x_1^{(0)},x_1^{(1)},\dots, x_{n_1}^{(1)}, x_1^{(2)},\dots,x_{n_2}^{(2)},\dots 
    \]
    where the distance between $x_{i}^{(j)}$ and $x_1^{(0)}$ is exactly $j$, for $1\leq i\leq n_j$. We define the \emph{rank} function to be
    \begin{align*}
        \rank: V&\to \mathbb{Z}\\
        x_i^{(j)}&\mapsto j.
    \end{align*}
    We will also view $H$ as a rooted tree with root at $x_1^{(0)}$. We remark that any vertex $x_i^{(j)}$ has a unique \emph{predecessor}, i.e., a vertex $x_{k}^{(j-1)}$ such that $x_{k}^{(j-1)}x_{i}^{(j)}$ is an edge of $H$.
    
    By definition, any $m\in \mingens(I(\mathcal{H}))$ can be written as
    \[
    m=x_{i_1}^{(j)}x_{i_2}^{(j+1)}\cdots x_{i_k}^{(j+k-1)}.
    \]
    Thus, we have a well-defined function
    \begin{align*}
        \min: \mingens(I(\mathcal{H}))&\to \mathbb{Z}\\
        m&\mapsto \min\{j\in \mathbb{Z} ~\big\vert~ x_i^{(j)}\mid m \text{ for some i}\}.
    \end{align*}
    Consider a total order $(\succ)$ on $\mingens(I(\mathcal{H}))$ where for any $m,m'\in \mingens(I(\mathcal{H}))$, $\min(m)<\min(m')$ implies that $m\succ m'$. We remark that there are multiple such total orders. We will show that $I(\mathcal{H})$ is bridge-friendly with respect to $(\succ)$. Due to~\cite[Lemma 2.9]{CHM24-first} and~\cite[Remark 2.12]{CHM24-first}, it suffices to show that whenever there exist $m_1,m_2,m_3\in \mingens(I(\mathcal{H}))$ such that
    \begin{itemize}
        \item $y\mid m_1,m_3$, $y\nmid m_2$, and
        \item $z\mid m_2,m_3$, $z\nmid m_1$
    \end{itemize}
    for some distinct vertices $y,z$ of $\mathcal{H}$, we have $m_3 \succ m_1$ or $m_3\succ m_2$. By the above definition, it suffices to show that under these hypotheses, we have $\min(m_3)< \min(m_1)$ or $\min(m_3)<\min(m_2)$. 
    
    Without loss of generality, as $\mathcal{H}$ is a rooted hypertree, we can assume $\rank y>\rank z$. Because each vertex has a unique predecessor in $H$,   if $x_i^{(j)}$ divides both $m_1$ and $m_3$ for some $i$ and $j$, then so does any vertex $x$ that divides $m_3$ and $\max\{\min(m_1),\min(m_3)\} \leq \rank x \leq j$. In particular, since $z\mid m_2,m_3$ and $z\nmid m_1$, we have 
    \[
    \max\{\min(m_2),\min(m_3)\}\leq \rank z< \max\{\min(m_1),\min(m_3)\}.
    \]
    Thus $\max\{\min(m_1),\min(m_3)\}\neq \min(m_3)$. In particular, this means $\min(m_3)< \min(m_1)$, as ~claimed.
\end{proof}

As immediate consequences of Theorem~\ref{thm:rooted-hypertree-friendly}, we recover the following results; see~\cite{BW02, CKW24} for necessary terminology.

\begin{corollary}[\protect{\cite[Theorem 3.17]{BW02} and \cite[Theorem 3.8]{CKW24}}]
    The path ideal of a path and edge ideal of a tree have a cellular minimal  resolution.
\end{corollary}

Edge ideals of rooted hypertrees also include the path ideals of \emph{rooted} trees considered in~\cite{BHO11}. The minimal free resolution of path ideals of rooted trees was described in~\cite{BHO11} using the mapping cone construction. Theorem~\ref{thm:rooted-hypertree-friendly} allows us to recover the minimal free resolutions of path ideals of rooted trees using discrete Morse theory, and thus has more implications.

\begin{corollary}
    The path ideal of a rooted tree has a cellular minimal resolution.
\end{corollary}

\begin{remark}
    Not all hypertrees are rooted. Indeed, one can check that the hypertree in Example~\ref{ex:hypertree-non-rooted} is not rooted. 
\end{remark}

Depending on the structure of the rooted hypertrees, one can also deduce (or recover) formulas for (total and graded) Betti numbers, for example, those given in~\cite[Theorem 3.17]{CK24} and~\cite[Theorem 2.7]{BHO11}. 


\bibliographystyle{amsplain}
\bibliography{refs}
\end{document}